\documentclass[10pt, reqno]{amsart}

\usepackage{esint,amsfonts,amsmath,amssymb,epsfig,stmaryrd,mathrsfs,bm,accents,yfonts,mathtools}
\usepackage{color,bm,epic,ifthen,eepic,dsfont,bbm}
\begingroup
\newtheorem{theorem}{Theorem}[section]
\newtheorem{lemma}[theorem]{Lemma}
\newtheorem{proposition}[theorem]{Proposition}

\endgroup

\theoremstyle{definition}
\newtheorem{definition}[theorem]{Definition}
\newtheorem{remark}[theorem]{Remark}

\numberwithin{equation}{section}
\newcommand{\eps}{{\varepsilon}}

\newcommand{\Lip}{{\rm {Lip}}}

\newcommand{\dist}{{\rm {dist}}}

\newcommand{\supp}{{\rm supp}\,}
\newcommand{\res}{\mathop{\hbox{\vrule height 7pt width .3pt depth 0pt
\vrule height .3pt width 5pt depth 0pt}}\nolimits}
\newcommand{\Id}{{\rm Id}\,}
\newcommand{\mass}{{\mathbf{M}}}

\newcommand{\Iqs}{{\mathcal{A}}_Q(\R^{n})}
\newcommand{\Iq}{{\mathcal{A}}_Q}
\def\a#1{\left\llbracket{#1}\right\rrbracket}
\newcommand{\abs}[1]{\left|#1\right|}
\newcommand{\norm}[2]{\left\|#1\right\|_{#2}}
\newcommand{\ra}{\right\rangle}
\newcommand{\la}{\left\langle}
\newcommand{\D}{\textup{Dir}}
\newcommand{\de}{\partial}

\newcommand{\etaa}{{\bm{\eta}}}
\newcommand{\ph}{\varphi}
\newcommand{\graph}{\mathop{graph}}

\newcommand{\cG}{{\mathcal{G}}}

\newcommand{\cH}{{\mathcal{H}}}

\newcommand\R{{\mathbb R}}
\newcommand\N{{\mathbb N}}

\newcommand\C{{\mathbb C}}

\newcommand\sD{{\mathscr D}}

\newcommand\sV{{\mathscr V}}

\title[Complex varieties and higher integrability]{Complex varieties and higher integrability of $\D$-minimizing
$Q$-valued functions}
\author[E.~N.~Spadaro]{Emanuele Nunzio Spadaro}
\address{Universit\"at Z\"urich}
\email{emanuele.spadaro@math.uzh.ch}

\begin{document}

\begin{abstract}
We provide new elementary proofs of the following two results:
every complex variety is locally the graphs of a $\D$-minimizing function,
first proved by Almgren \cite{Alm};
the gradients of $\D$-minimizing functions,
in principle square-summable, are $p$-integrable for some $p>2$,
proved by De Lellis and the author \cite{DLSp2}.
In the planar case, we prove that our integrability exponents are optimal.
\end{abstract}

\maketitle

\section{Introduction}
Almgren developed the theory of $\D$-minimizing multi-valued
functions in his big regularity paper \cite{Alm} as a first
step toward the regularity of area-minimizing currents in
codimension bigger than $1$.
Following the pioneering ideas of De Giorgi, the starting point was
the approximation of minimal currents via harmonic functions,
which are the minimizers of the first non-constant term in the expansion
of the area functional: the Dirichlet energy.
However, due to the unavoidable phenomenon of branching points
as, for example, in the area-minimizing currents induced by complex varieties,
he needed to develop the theory of $\D$-minimizing $Q$-valued functions,
that are multi-valued functions minimizing a suitable Dirichlet energy.

In this paper, following the work in \cite{DLSp},
we address two questions on Almgren's $Q$-valued functions:
we show that complex varieties are locally graphs of $\D$-minimizing functions and
prove the higher integrability of the gradient of a $\D$-minimizing $Q$-function.
\begin{theorem}\label{t:complex}
Let $\mathscr{V}\subseteq \C^\mu\times\C^\nu\simeq\R^{2\mu}\times\R^{2\nu}$
be an irreducible holomorphic variety
which is a $Q:1$-cover of the ball $B_2\subseteq\C^\mu$
under the orthogonal projection.
Then, there exists a $\D$-minimizing $Q$-valued function
$f\in W^{1,2}(B_1,\Iq(\R^{2\nu}))$ such that
$\graph(f)=\mathscr{V}\cap (B_1\times\C^\nu)$.
\end{theorem}
\begin{theorem}\label{t:higher}
There exists $p=p(n,m,Q)>2$ such that, for every $\Omega\subseteq\R^m$ open and
$u\in W^{1,2}(\Omega,\Iqs)$ $\D$-minimizing, $|Du|\in L_{loc}^p(\Omega)$.
\end{theorem}
Theorem \ref{t:complex} provides many examples of $\D$-minimizing functions and,
in particular, shows that the H\"older continuity and the estimate of the singular set
of a $\D$-minimizer proved in \cite{Alm} and \cite{DLSp} are optimal results.
Theorem \ref{t:complex} has been proved by Almgren in his big regularity paper
\cite[Theorem 2.20]{Alm} using a deep and complicated approximation theorem
of minimal currents via graphs of Lipschitz $Q$-functions
(see also \cite{DLSp2}).
Here we give a more elementary proof avoiding the approximation result by Almgren.
Moreover, for the planar case we also provide an alternative argument which exploits
the equality between the area and the energy of conformal maps.
We hope that this approach can be extended to the study of
regularity issues for more complicated calibrated geometries.

Theorem \ref{t:higher} has been first proved by
De Lellis and the author in \cite{DLSp2} in connection with a new
higher integrability estimate for minimal currents and it plays a crucial role in
the proof of Almgren's approximation theorem given there.
Here, we propose a different ``intrinsic'' proof, where ``intrinsic'' means
based only on the metric theory of $Q$-valued functions as developed in \cite{DLSp}.
In case $m=2$, we can exploit the fact that $\D$-minimizing functions have isolated singularities
(proven in \cite{DLSp}) to find the optimal integrability. The optimality is indeed
shown by the examples provided by complex varieties in the first part of the paper.

\medskip

The paper is organized as follows. In Section \ref{s:Qfunc} we collect some basic results
and definitions on $Q$-valued functions and the rectifiable currents supported by their graphs.
In Section \ref{s:complex} we identify complex varieties as graphs of
Sobolev $Q$-valued functions and prove Theorem \ref{t:complex}.
Finally, Section \ref{s:higher} contains the proof of Theorem \ref{t:higher} which passes through
a Caccioppoli and a reverse H\"older inequality for $\D$-minimizing functions.

\subsection{Acknowledgements} The author is grateful to Camillo De Lellis
for many stimulating discussions. This research has been supported by the
Forschungskredit of the University of Zurich.

\section{$Q$-valued functions}\label{s:Qfunc}
In what follows, we adopt the notation and the approach introduced in \cite{DLSp},
which differs from Almgren's original one.
For the definitions of the metric space of $Q$-points $(\Iq,\cG)$,
Sobolev $Q$-valued function and Dirichlet energy, we refer to \cite{DLSp}.
We say that a function $f:\Omega\subset\R^{m}\to\Iqs$ has a smooth \textit{local selection}
in $\Omega'\subseteq\Omega$ if, for every $x\in\Omega'$,
there exist $r>0$ and $f_i:B_r(x)\to\R^{n}$ smooth functions
such that $f\vert_{B_r(x)}=\sum_{i=1}^Q\a{f_i}$.
Note that, in this case, $|Df|^2=\sum_i|Df_i|^2$ is well defined on the whole $\Omega'$.
We observe the following simple consequence of the definition, which for reader's
convenience we state as a lemma.
\begin{lemma}\label{l:Sob}
Let $f:\Omega\subset\R^{m}\to\Iq$ have a smooth local selection in
$\Omega'\subseteq\Omega$.
If $\dim_{\cH}(\Omega\setminus \Omega')\leq m-2$ and
$\int_{\Omega'}|Df|^2<+\infty$,
then $f$ belongs to $W^{1,2}(\Omega,\Iq)$.
\end{lemma}
\begin{proof}
The proof follows from the characterization
of classical Sobolev functions via the slice property.
Indeed, for every $T\in \Iq$, the function $x\mapsto \cG(f(x),T)$ is
smooth and satisfies
$|D(\cG(f(\cdot),T))|\leq |Df|$ in $\Omega'$ (cp.~to \cite[Proposition 2.17]{DLSp}).
Therefore, since the projection of $\Omega\setminus\Omega'$ on each coordinate
hyperplane is a set of $\cH^{m-1}$ measure zero, for
$\cH^{m-1}$-a.e.~line $l$ parallel to the axes,
the restriction of $\cG(f(\cdot),T)$ to $l$ belongs to $W^{1,2}$.
Recalling \cite[Section 4.9.2]{EG},
it follows that $\cG(f(\cdot),T)\in W^{1,2}(\Omega)$ with
$|D(\cG(f(\cdot),T))|\leq |Df|$ a.e.~in $\Omega$.
Hence, by the definition of Sobolev $Q$-functions \cite[Definition 0.5]{DLSp}, we conclude.
\end{proof}

We will need also a technical result about the lower semicontinuity of the $L^p$ norm of
the gradient under weak convergence.
Although this is a special case of the result in \cite{DLFS}, we include here
an elementary proof for the sake of completeness.

\begin{lemma}[Semicontinuity]\label{l:sc}
Let $f_k$, $f\in W^{1,p}(\Omega,\Iq)$, $p<\infty$, be such that
$\lim_{k}\int_\Omega\cG(f_k,f)^p=0$ and
$\sup_k\int_\Omega |Df_k|^p<\infty$. Then,
\begin{equation}\label{e:sc}
\int_\Omega|Df|^p\leq\liminf_{k\to+\infty}\int_\Omega|Df_k|^p.
\end{equation}
\end{lemma}
\begin{proof}
The proof of this result is very similar to the proof of the semicontinuity for the Dirichlet
energy given in \cite[Section 2.3.2]{DLSp}.
Let $\{T_l\}_{l\in\N}$ be any dense subset of $\Iq$ and recall that
by \cite[Proposition 4.2]{DLSp}
$|Df|$ is the monotone limit of $h_N$ with
$$
h_N^2=\max_{l_j\leq N}\sum_j\bigl(\partial_j \cG (f, T_{l_j})\bigr)^2.
$$
By the Monotone Convergence Theorem, $\int |Df|^p = \sup_N \int h_{N}^p$.
Therefore, denoting by $\mathcal{P}_{N^m}$
the collections $P=\{E_{\bar l}\}_{\bar l=\{l_1,\ldots,l_m\}\in N^m}$
of $N^m$ disjoint open subsets of $\Omega$, as in \cite{DLSp} we conclude that
\begin{equation}\label{e:rappresenta}
\int_\Omega|Df|^p = \sup_N \int_\Omega h_{N}^p
=\sup_N \sup_{P\in \mathcal{P}_{N^m}}
\sum_{E_{\bar l}\in P} \int_{E_{\bar l}}
\left(\sum_j\bigl(\partial_j \cG (f, T_{l_j})\bigr)^2\right)^{\frac{p}{2}}.
\end{equation}
It follows easily from the hypotheses that,
for every $\bar l=\{l_1,\ldots,l_m\}$ and every open set $E_{\bar l}$,
the vector-valued maps $(\partial_1\cG(f_k,T_{l_1}),\ldots,\partial_m\cG(f_k,T_{l_m}))$ converge weakly in $L^p(E_{\bar l})$ to
$(\partial_1\cG(f,T_{l_1}),\ldots,\partial_m\cG(f,T_{l_m}))$.
Hence, by the semicontinuity of the norm,
\begin{equation*}\label{e:semi2}
\int_{E_{\bar l}}
\left(\sum_j\bigl(\partial_j \cG (f, T_{l_j})\bigr)^2\right)^{\frac{p}{2}}
\leq \liminf_{k\to+\infty}
\int_{E_{\bar l}}
\left(\sum_j\bigl(\partial_j \cG (f_k, T_{l_j})\bigr)^2\right)^{\frac{p}{2}}.
\end{equation*}
Summing in 
$E_l\in P$,
in view of \eqref{e:rappresenta}, we achieve \eqref{e:sc}.
\end{proof}

The main regularity results for $\D$-minimizing $Q$-valued functions
are collected in the following theorem
(see \cite[Theorems 0.9 and 0.11]{DLSp}).
In order to state them, we recall the definition of regular and singular points.
\begin{definition}\label{d:regular}
A $Q$-valued function $f$ is regular at a point $x\in \Omega$ if there exist
a neighborhood $U$ of $x$ and $Q$ analytic functions $f_i:U\to \R^{n}$ such that
$f \vert_U= \sum_i \a{f_i}$ and either $f_i (y)\neq f_j (y)$ for every $y\in U$ or
$f_i\equiv f_j$.
The singular set $\Sigma_f$ of $f$ is the complement of the set of regular points.
\end{definition}

\begin{theorem}\label{t:regularity}
For every $\D$-minimizing $f\in W^{1,2} (\Omega,\Iq)$ the following holds: 
\begin{itemize}
\item[$(i)$] there exists $\alpha=\alpha (m,Q)>0$ ($\alpha (2, Q)=1/Q$) such that
$f\in C^{0,\alpha} (\Omega')$ for every $\Omega'\subset\subset\Omega$ and
\begin{equation}\label{e:dec}
\hspace{1.5cm}\D(f,B_r(x))\leq\left(\frac{r}{\rho}\right)^{2\alpha}\D(f,B_\rho(x)),
\quad\forall\;r\leq \rho\;\text{with}\; B_\rho(x)\subseteq\Omega;
\end{equation}
\item[$(ii)$] the Hausdorff dimension of $\Sigma_f$ is at most $m-2$ and,
if $m=2$, $\Sigma_f$ consists of isolated points.
\end{itemize}
\end{theorem}

\subsection{Push-forward of currents under $Q$-functions}\label{s:current}
We define now the integer rectifiable current associated to the graph of a $Q$-valued function.

Given a $Q$-valued function $f:\R^{m}\to\Iqs$, we set $\bar f=\sum_i\a{(x,f_i(x))}$,
$\bar f:\R^{m}\to\Iq(\R^{m+n})$.
If $R\in \sD_k(\R^m)$ is a rectifiable current associated to a $k$-rectifiable set $M$
with multiplicity $\theta$,
$R=\tau(M,\theta,\xi)$, where $\xi$ is a borel simple $k$-vector field orienting $M$
(we use the notation in \cite{Sim}),
and if $f$ is a proper Lipschitz $Q$-valued function,
we can define the push-forward of $T$ under $f$ as follows.
\begin{definition}\label{d:TR}
Given $R=\tau(M,\theta,\xi)\in \sD_k(\R^m)$ and $f\in\Lip(\R^{m},\Iqs)$ as above,
we denote by $T_{f,R}$ the current in $\R^{m+n}$ defined by
\begin{equation}\label{e:TR}
\la T_{f, R}, \omega\ra = \int_{M} \theta \,\sum_i\la
\omega\circ\bar f_i, D^M\bar f_{i\#}\xi\ra d\,\cH^k
\quad\forall\;\omega\in\sD^k(\R^{m+n}),
\end{equation}
where $\sum_i \a{D^M\bar f_i(x)}$ is the differential of $\bar f$ restricted to $M$.
\end{definition}
\begin{remark}
Note that, by Rademacher's theorem \cite[Theorem 1.13]{DLSp}
the derivative of a Lipschitz $Q$-function is defined a.e.~on smooth manifolds and, hence,
also on rectifiable sets.
\end{remark}

As a simple consequence of the Lipschitz decomposition
in \cite[Proposition 1.6]{DLSp},
there exist $\{E_j\}_{j\in\N}$ closed subsets of $\Omega$,
positive integers $k_{j,l},\,L_j\in\N$ and Lipschitz functions $f_{j,l}:E_j\to\R^{n}$, for $l=1,\ldots,L_j$, such that
\begin{equation}\label{e:Lip decomp}
\cH^{k}(M\setminus \cup_{j} E_j)=0\quad\text{and}\quad
f\vert_{E_j}=\sum_{l=1}^{L_j}k_{j,l}\,\a{f_{j,l}}.
\end{equation}
From the definition, $T_{f,R}=\sum_{j,l}k_{j,l}\bar f_{j,l\#}(R\res E_{j})$
is a sum of rectifiable currents defined by the push-forward under single-valued Lipschitz functions.
Therefore, it follows that $T_{f,R}$ is rectifiable and coincides with
$\tau\big(\bar f(M),\theta_{f}, \vec{T}_{f}\big)$,
where
\[
\theta_{f}(x,f_{j,l}(x))=k_{j,l}\theta(x)\quad \text{and}\quad
\vec{T}_{f}(x,f_{j,l}(x))=\frac{D^M\bar f_{j,l\#}\xi(x)}{|D^M\bar f_{j,l\#}\xi(x)|}
\quad\forall\;x\in E_j.
\]
By the standard area formula, using the above decomposition of $T_{f,R}$,
we get an explicit expression for the mass of $T_{f,R}$:
\begin{equation}\label{e:mass TR}
\mass\left(T_{f, R}\right)=\int_{M}|\theta|\,
\sum_i\sqrt{\det\left(D^M\bar f_i\cdot (D^M\bar f_i)^T\right)}\,d\,\cH^k.
\end{equation}

With a slight abuse of notation, when $R=\a{\Omega}\in \sD_m(\R^m)$
is given by the integration over a Lipschitz domain $\Omega\subset\R^m$ of the standard
$m$-vector $\vec e=e_1\wedge\cdots\wedge e_m$, we write simply $T_{f,\Omega}$ for $T_{f,R}$.
The same we do for $T_{f,\de \Omega}$, understanding that $\de\Omega$ is oriented as the
boundary of $\a{\Omega}$.
The main result for what concerns the push-forward under $Q$-valued functions
is given in the following theorem proven in \cite[Theorem C.3]{DLSp2}.

\begin{theorem}\label{t:de Tf}
For every $\Omega$ Lipschitz domain and $f\in \Lip(\Omega,\Iq)$,
$\de \,T_{f,\Omega}=T_{f,\de \Omega}$.
\end{theorem}

Up to now we have defined the push-forward under Lipschitz maps.
Nevertheless, thanks to the approximate differentiability property of Sobolev
$Q$-functions (see \cite[Corollary 2.7]{DLSp}), for full dimensional current $R=\a{\Omega}$,
the definition of $T_{f,\Omega}$ in \eqref{e:TR} makes sense for Sobolev functions
as soon as the action is finite for every differential form $\omega\in\sD^m(\R^{m+n})$.
It is easy to verify that this condition is satisfied if
\begin{equation*}\label{e:finite mass}
\mass(T_{f,\Omega})=\int_{\Omega}\sum_i\sqrt{\det\left(D^M\bar f_i\cdot (D^M\bar f_i)^T\right)}<+\infty.
\end{equation*}
For such functions, we have the following Taylor expansion of the mass of $T_{f,\Omega}$.

\begin{lemma}\label{l:mass vis dir}
Let $f\in W^{1,2}(\Omega,\Iq)$ such that $\mass\left(T_{f,\Omega}\right)<+\infty$.
Then,
\begin{equation}\label{e:mass vis dir1}
\mass\left(T_{\lambda f,\Omega}\right)=Q\,|\Omega|+\frac{\lambda^2}{2}\,\D(f,\Omega)+o\left(\lambda^2\right)
\quad\text{as}\quad\lambda\to0.
\end{equation}
\end{lemma}
\begin{proof}
For every $\lambda>0$, set
$A_\lambda=\big\{|Df|\leq\lambda^{-\frac{1}{2}}\big\}$
and $B_\lambda=\big\{|Df|>\lambda^{-\frac{1}{2}}\big\}$.
Since $f\in W^{1,2}(\Omega,\Iq)$, for $\lambda\to0$, we have that
\begin{equation}\label{e:su A}
\D(\lambda\,f,\Omega)=
\D(\lambda\,f,A_\lambda)+\lambda^2\int_{B_\lambda}|Df|^2
=\D(\lambda\,f,A_\lambda)+o\left(\lambda^2\right).
\end{equation}
Using the inequality $\sqrt{1+x^2}\geq 1+\frac{x^2}{2}-\frac{x^4}{4}$ for $|x|\leq 2$,
since $\lambda\,|Df|\leq\sqrt{\lambda}$ in $A_\lambda$, for $\lambda\leq 4$ we infer that
\begin{align}\label{e:est mass1}
\mass\left(T_{\lambda f,\Omega}\right)
&\geq
\sum_i\int_{\Omega}\sqrt{1+\lambda^2\,|Df_i|^2}
\geq  Q\,|B_\lambda|+\int_{A_\lambda}
\left(1+\frac{\lambda^2\,|Df|^2}{2}-C\,\lambda^4\,|Df|^4
\right)\notag\\
&\geq Q\,|\Omega|+\frac{\lambda^2}{2}\,\D(f,A_\lambda)-
\int_{A_\lambda}C\,\lambda^3\,|Df|^2\notag\\
&
\hspace{-0.15cm}\stackrel{\eqref{e:su A}}{=}
Q\,|\Omega|+\frac{\lambda^2}{2}\,\D(f,\Omega)+o\left(\lambda^2\right).
\end{align}

For what concerns the reversed inequality, we argue as follows.
In $A_\lambda$, since for every multi index $\alpha$ with $|\alpha|\geq2$ we have
\[
\lambda^{2|\alpha|}|M_{f_i}^\alpha|^2\leq C\,\lambda^{2|\alpha|}|Df_i|^{2|\alpha|}\leq
C\, \lambda^3|Df_i|^2,
\]
we use the inequality $\sqrt{1+x^2}\leq1+\frac{x^2}{2}$ and get
\begin{align}\label{e:est massA}
\mass\big(T_{\lambda f,A_\lambda}\big)
&\leq
\sum_i\int_{A_\lambda}\sqrt{1+\lambda^2\,|Df_i|^2+C\,\lambda^3\,|Df_i|^2}\notag\\
&=
Q\,|A_\lambda|+\frac{\lambda^2}{2}\,\D(f,A_\lambda)+o\left(\lambda^2\right).
\end{align}
In $B_\lambda$, instead, we use the same inequality and
the condition $\mass(T_{f,\Omega})<+\infty$ to infer
\begin{align}\label{e:est massB}
\mass\big(T_{\lambda f,B_\lambda}\big)
&\leq 
\sum_i\int_{B_\lambda}\sqrt{1+\lambda^2\,|Df_i|^2}+\sqrt{\sum_{|\alpha|\geq 2}\lambda^{2|\alpha|}{M_{f_i}^\alpha}^2}\notag\\
&\leq 
Q\,|B_\lambda|+\frac{\lambda^2}{2}\,\D(f,B_\lambda)+
\sum_i\int_{B_\lambda}\lambda^2\sqrt{\sum_{|\alpha|\geq 2}{M_{f_i}^\alpha}^2}\notag\\
&\hspace{-0.15cm}\stackrel{\eqref{e:su A}}{\leq} Q\,|B_\lambda|+o(\lambda^2)
+\lambda^2\,\mass(T_{f,B_\lambda})=Q\,|B_\lambda|+o\left(\lambda^2\right).
\end{align}
From \eqref{e:est mass1}, \eqref{e:est massA} and \eqref{e:est massB}, the proof follows.
\end{proof}

\section{Complex varieties and $\D$-minimizing functions}\label{s:complex}
\subsection{Complex varieties as minimal currents}
In the following we consider irreducible holomorphic varieties $\sV\subseteq\C^{\mu+\nu}$
of dimension $\mu$.
Following Federer \cite{Fed2}, we associate to $\sV$ the
integer rectifiable current of real dimension $2\mu$ denoted by $\a{\mathscr{V}}$
given by the integration over the manifold part of $\sV$, $\sV_{\rm{reg}}$.
Recall that the singular part $\sV_{\rm{sing}}=\sV\setminus \sV_{reg}$ is a
complex variety of dimension at most $(\mu-1)$.
A well-known result by Federer asserts that $\a{\sV}$ is a mass-minimizing cycle.

\begin{theorem}\label{t:fed}
Let $\mathscr{V}$ be an irreducible holomorphic variety.
Then, the integer rectifiable current $\a{\mathscr{V}}$ has locally finite mass and is
a locally mass-minimizing cycle,
that means $\de \a{\mathscr{V}}=0$
and $\mass(\a{\mathscr{V}})\leq\mass(S)$ for every integer
current $S$ with $\de S=0$ and $\supp(S-\a{\mathscr{V}})$ compact.
\end{theorem}

We consider domains
$\Omega\subseteq\R^{2\mu}\simeq\C^\mu$ with the usual identification
$(x_l,y_l)\simeq z_l=(x_l+iy_l)$ for $l=1,\ldots,\mu$.
Moreover, $\mathscr{V}\subseteq\Omega\times\R^{2\nu}\subseteq\R^{2\mu+2\nu}\simeq\C^{\mu+\nu}$
is always supposed to be a $Q:1$-cover of $\Omega$ under the orthogonal projection $\pi$
onto $\Omega$, that is $\pi_{\#}\a{\sV}=Q\a{\Omega}$.

Clearly, under this hypothesis, there exists a $Q$-valued function
$f:\Omega\to\Iq(\R^{2\nu})$ such that $\mathscr{V}=\graph(f)$.
From Definition \ref{d:regular}, we readily deduce
$\Sigma_f\subseteq\pi(\sV_{\rm{sing}})$,
which in particular implies $\dim_{\cH}(\Sigma_f)\leq 2\mu-2$.
Therefore, locally in $\Omega\setminus \Sigma_f\times\R^{2\nu}$,
$\sV$ is the superposition of graphs of holomorphic functions, that is,
for every $w\in\Omega\setminus \Sigma_f$, there exist a radius $r$ and $Q$ holomorphic
functions $f_i:B_r(w)\to\C^\nu$ such that $f\vert_{B_r(w)}=\sum_i\a{f_i}$.
The following are the main properties of $f$.

\begin{proposition}\label{p:fSob}
Let $\sV\subseteq \Omega\times\R^{2\nu}$ be a holomorphic variety as above and $f$
the associated $Q$-valued function.
Then, the following holds:
\begin{itemize}
\item[$(i)$] $f\in W^{1,2}(\Omega,\Iq)$ and, for $\mu=1$,
$M(\a{\mathscr{V}}\res \Omega)=Q+\frac{\D(f,\Omega)}{2}$;
\item[$(ii)$] $\a{\sV}\res\Omega=T_{f,\Omega}$ and
$\de (\a{\sV}\res B_r(x))=T_{f,\de B_r(x)}$
for every $x$ and a.e.~$r>0$ with $B_r(x)\subseteq\Omega$.
\end{itemize}
\end{proposition}

\begin{proof}
Note that, for every smooth $h:\R^2\to\R^{2\nu}$
and, as usual, $\bar h(w)=(w,h(w))$,
\begin{equation}\label{e:area<energia}
\sqrt{\det\left(D\bar h\cdot D\bar h^T\right)}\leq 1+\frac{|Dh|^2}{2},
\end{equation}
with equality if and only if $h$ is conformal,
i.e.~$|\de_{x} h|=|\de_{y} h|$ and $\de_{x} h\cdot \de_{y} h=0$.
Indeed, \eqref{e:area<energia} reads as
\begin{align*}\label{e:Jh}
\det\left(D\bar h\cdot D\bar h^T\right)
&= 
\det\left(\begin{array}{cc}
1+|\de_x h| &  \de_x h\cdot \de_y h \\
\de_x h\cdot \de_y h & 1+|\de_y h|
\end{array} \right)
\leq \left(1+\frac{|\de_xh|^2+|\de_y h|^2}{2}\right)^2,
\end{align*}
which in turn is equivalent to
$0\leq \left(|\de_x h|^2-|\de_y h|^2\right)^2+
4(\de_x h\cdot \de_y h)^2$. 

In the case $\mu=1$,
applying \eqref{e:area<energia} to the local holomorphic, hence conformal, selection of $f$,
from \eqref{e:mass TR} we get
\begin{equation}\label{e:mass=dir}
M(\a{\mathscr{V}}\res(\Omega\setminus\Sigma_f))=
Q+\frac{\D(f,\Omega\setminus\Sigma_f)}{2}.
\end{equation}
In the case $\mu>1$ and $g:\R^{2\mu}\to \R^{2\nu}$ smooth,
\eqref{e:area<energia} together with
Binet--Cauchy's formula (see \cite[Section 3.2 Theorem 4]{EG}),
for every $l=1,\cdots,\mu$, we infer
\begin{align}\label{e:mu>1}
\det\left(D\bar g\cdot D\bar g^T\right)&=
1+|Dg|^2+\sum_{|\alpha|=|\beta|\geq2}M_{\alpha\beta}(Dg)^2\notag\\
&\geq
1+|\de_{x_l}g|^2+|\de_{y_l}g|^2+
\sum_{i,j=1}^{2\nu}(\de_{x_l}g^i\de_{y_l}g^j-\de_{x_l}g^j\de_{y_l}g^i)^2\notag\\
&=\det\left(\nabla_l\bar g\cdot \nabla_l\bar g^T\right),
\end{align}
where $M_{\alpha\beta}$ stands for the $\alpha,\beta$ minors of a matrix
and $\nabla_l$ denotes the derivative with respect to $x_l$ and $y_l$.
Hence, if $f_i$ is a local holomorphic, consequently conformal,
selection for $f:\Omega\subset \R^{2\mu}\to\Iq$, we infer that
\begin{align*}
\mu\,Q+\frac{|Df|^2}{2}&=\sum_{i=1}^Q\sum_{l=1}^\mu
\left(1+\frac{|\nabla_l f_i|^2}{2}\right)
\stackrel{\eqref{e:area<energia}}{=}
\sum_{i=1}^Q\sum_{l=1}^\mu\sqrt{\det\left(\nabla_l\bar f_i\cdot \nabla_l\bar f_i^T\right)}\notag\\
&\stackrel{\eqref{e:mu>1}}{\leq}
\mu \sum_{i=1}^Q\sqrt{\det\left(D\bar f_i\cdot D\bar f_i^T\right)}.
\end{align*}
Integrating, we conclude, for $\mu>1$,
\begin{equation}\label{e:mass=dir2}
M(\a{\mathscr{V}}\res(\Omega\setminus\Sigma_f))\geq
Q+\frac{\D(f,\Omega\setminus\Sigma_f)}{2\,\mu}.
\end{equation}
Now since the mass of $\a{\sV}$ is finite,
by \eqref{e:mass=dir} and \eqref{e:mass=dir2} the energy of $f$ is finite in
$\Omega\setminus\Sigma_f$.
Being $\dim_{\cH}(\Sigma_f)\leq m-2$, Lemma \ref{l:Sob} gives $(i)$.

Being $\a{\sV}$ defined by the integration over $\sV_{\rm{reg}}$
and $\cH^m(\pi(\sV_{\rm{sing}}))=0$,
it follows straightforwardly that $T_{f,\Omega}$ is well-defined by \eqref{e:TR}
and coincides with $\a{\sV}$.
For the same reason, since also $\cH^{m-1}(\pi(\sV_{\rm{sing}}))=0$,
 $\de(\a{\sV}\res B_r(x))=T_{f,\de B_r(x)}$
for every $B_r(x)\subseteq \Omega$ such that
$f\vert_{\de B_r(x)}\in W^{1,2}$ and $\mass(\de(\a{\sV}\res B_r(x)))$ is finite,
that is for every $x$ and a.e.~$r>0$, thus concluding the proof of $(ii)$.
\end{proof}

\subsection{Proof of Theorem \ref{t:complex}}
Now we are ready to prove the first main result of the paper.
We divide the proof in two parts: in the first one we give an argument for the planar case
which is particularly simple and exploits the equality between the area and the energy functionals;
in the second part we give a proof valid in every dimension.
\subsubsection{Planar case $\mu=1$}
In view of Proposition \ref{p:fSob}, we need only to show that
$f$ is $\D$-minimizing in $B_1$.
Choose a radius $r\in[1,2]$ such that $\de B_r\cap\Sigma_f =\emptyset$
and set $g=f\vert_{\de B_r}$. Note that $g$ is Lipschitz continuous.
For every $h\in \Lip(B_r,\Iq)$ with $h\vert_{\de B_r}=g$, 
from the Taylor expansion of the mass and from \eqref{e:area<energia}, we infer that
\begin{equation}\label{e:massineq}
\mass (T_{h,B_r}) - Q\leq \frac{\D(h,B_r)}{2}.
\end{equation}
By Theorem \ref{t:de Tf}, $\de T_{h,B_r}=T_{f,\de B_r}=\de(\a{\sV}\res B_r)$.
So, using Theorem \ref{t:fed} we infer
\begin{equation*}
\D(f,B_r)\stackrel{\eqref{e:mass=dir}}{=}
2 \left(\mass(T_{f,B_r})-Q\right)\leq
2 \left(\mass(T_{h,B_r})-Q\right)\stackrel{\eqref{e:massineq}}{\leq}
\D(h,B_r).
\end{equation*}
Since the set of Lipschitz functions with trace $g$ is dense in  $W_g^{1,2}(B_r,\Iq)$
(see \cite[Section 14]{DLSp}), this implies that $f$ is $\D$-minimizing in $B_r$ and, a fortiori,
in $B_1$.\qed

\begin{remark}
The planar result provides examples of $\D$-minimizing functions
with singular set of dimension $m-2$ for every $m$, thus proving the optimality of the
regularity Theorem \ref{t:regularity}.
Indeed, if $g:B_1\subseteq\R^2\to \Iq$ is $\D$-minimizing and $\Sigma_g\neq \emptyset$,
then $f:B_1\times \R^{m-2}\to\Iq$ with $f(x_1, x_2, \ldots, x_m) = g (x_1, x_2)$
is also $\D$-minimizing (see \cite[Lemma 3.24]{DLSp}) and
$\dim_{\cH}(\Sigma_f)=m-2$.
\end{remark}

\subsubsection{General case $\mu\geq1$}
Here we exploit the expansion of the mass given in Lemma \ref{l:mass vis dir}.
The reason why this can be done without the strong approximation
theory developed by Almgren in \cite{Alm} and reproved with different
methods in \cite{DLSp2} is that, given as above a complex variety which is the
graph of a multi-valued function, the rescaled current
$L_{\lambda\#}\a{\sV}=T_{\lambda f}$, where
$L_{\lambda}:\C^{\mu+\nu}\to\C^{\mu+\nu}$ is given by
$L_\lambda(x,y)=(x,\lambda y)$,
is also a complex variety (being the $L_\lambda$'s linear complex maps),
and, hence, it is also area-minimizing.

The proof is by contradiction. Assume $f$ is not $\D$-minimizing in $B_1$.
Then, there exists $u\in W^{1,2}(B_1,\Iq)$ and $\eta>0$ such that
$\D(u,B_1)\leq \D(f,B_1)-\eta$ and $u\vert_{\de B_1}=f\vert_{\de B_1}$.
Set
\[
w=
\begin{cases}
u & \text{in}\; B_1,\\
f & \text{in}\; B_2\setminus B_1.
\end{cases}
\]

We want to use $w$ in order to construct competitor currents for $L_{\lambda\#}\a{\sV}$.
To this aim, consider for every $\eps>0$ the Lipschitz approximations $w_\eps$ given by
(see \cite[Proposition 4.4]{DLSp}). It enjoys the following properties:
\begin{itemize}
\item[(a)] $|E_\eps|=o\left(\eps^2\right)$ as $\eps\to 0$,
where $E_\eps=\big\{w_\eps\neq w\big\}$;
\item[(b)] $\Lip(w_\eps)\leq \eps^{-1}$;
\item[(c)] $\norm{|Dw_\eps|-|Dw|}{L^2}=o(1)$ as $\eps\to 0$.
\end{itemize}
By Proposition \ref{p:fSob} and Lemma \ref{l:mass vis dir},
for every open $A$ such that $E_\eps\subseteq A$ and $|A|\leq 2|E_\eps|$,
\begin{align*}
\mass\Big(L_{\lambda\,\#}\big(\a{\sV}\res(E_\eps\times\R^{2\nu})
\big)\Big)
&=\mass\left(T_{\lambda f,E_\eps}\right)
\leq
\mass\left(T_{\lambda f,A}\right)\notag\\
&\hspace{-0.15cm}\stackrel{\eqref{e:mass vis dir1}}{=}
Q\,|A|+\frac{\lambda^2}{2}\int_A|Df|^2+o\left(\lambda^2\right)
=
o\left(\eps^2\right)+O\left(\lambda^2\right).
\end{align*}
Using Fubini's theorem and again Proposition \ref{p:fSob}, we can find radii $r_{\lambda,\eps}$ such that
\begin{equation}\label{e:slice}
\abs{E_\eps\cap \de B_{r_{\lambda,\eps}}}=o\left(\eps^2\right),
\end{equation}
\begin{equation}\label{e:slice2}
\de \big( L_{\lambda\,\#}\a{\sV}\res B_r\big)=T_{\lambda f,\de B_r}
\quad\text{and}\quad
\mass\left(T_{\lambda f, E_\eps\cap \de B_r}\right)=
o\left(\eps^2\right)+O\left(\lambda^2\right).
\end{equation}
Set $S_{\lambda\,\eps}=T_{\lambda f,\de B_{r_{\lambda,\eps}}}-T_{\lambda w_\eps,\de B_{r_{\lambda,\eps}}}$.
Note that, by Theorem \ref{t:de Tf}, being $w_\eps$ Lipschitz,
\[
\de S_{\lambda\,\eps}=\de T_{\lambda f,\de B_{r_{\lambda,\eps}}}-\de T_{\lambda w_\eps,\de B_{r_{\lambda,\eps}}}
\stackrel{\eqref{e:slice2}}{=}
\de\de \big( L_{\lambda\#}\a{\sV}\res B_r\big)=0.
\]
Moreover, since $\Lip(\lambda \,w_\eps)\leq\lambda\,\eps^{-1}$ and
$T_{\lambda f,\de B_{r_{\lambda,\eps}}\setminus E_\eps}=
T_{\lambda w_\eps,\de B_{r_{\lambda,\eps}}\setminus E_\eps}$,
the mass of $S_{\lambda\,\eps}$ can be estimated in the following way:
\begin{align}\label{e:mass bound rac}
\mass\left(S_{\lambda\,\eps}\right)&=
\mass\big(T_{\lambda f,E_\eps\cap\de B_{r_{\lambda,\eps}}}\big)
+\mass\big(T_{\lambda w_\eps,E_\eps\cap\de B_{r_{\lambda,\eps}}}\big)\notag\\
&\stackrel{\eqref{e:slice2}}{\leq}
o\left(\eps^2\right)+O\left(\lambda^2\right)+C\,\frac{\lambda\,|E_\eps|}{\eps}
\stackrel{\eqref{e:slice}}{\leq}
o\left(\eps^2\right)+O\left(\lambda^2\right)+o\left(\lambda\,\eps\right).
\end{align}
For $\eps=\lambda$, $\mass\left(S_{\lambda\,\lambda}\right)=O\left(\lambda^2\right)$
and, by the isoperimetric inequality \cite[Theorem 30.1]{Sim},
there exists an integer rectifiable current $R_{\lambda}$ such that
\begin{equation}\label{e:mass racc}
\de R_\lambda=S_{\lambda\,\lambda}
\quad\text{and}\quad
\mass\left(R_\lambda\right)\leq \mass\left(S_{\lambda\,\lambda}\right)^{\frac{m}{m-1}}
=o\left(\lambda^2\right).
\end{equation}

The current $T_\lambda=T_{\lambda\,w_\lambda,B_{r_{\lambda}}}+R_\lambda$
contradicts now the minimality of the complex current
$L_{\lambda\,\#}(\a{\sV}\res B_{r_\lambda})$.
Indeed, it is easy to verify that $\de T_\lambda=\de(L_{\lambda\,\#}\a{\sV}\res B_{r_\lambda})$ and, for small $\lambda$,
\begin{align*}
\mass\left(T_\lambda\right)-
\mass\left(L_{\lambda\,\#}\a{\sV}\res\left(B_{r_\lambda}\times\R^{2\nu}\right)\right)=&{}
Q\,|B_{r_\lambda}|+\frac{\lambda^2}{2}\,\D(w_\lambda,B_{r_\lambda})+\notag\\
&-Q\,|B_{r_\lambda}|-\frac{\lambda^2}{2}\,\D(f,B_{r_\lambda})+o\left(\lambda^2\right)
\\
\leq &{} 
-\frac{\lambda^2\,\eta}{4}+o\left(\lambda^2\right)< 0.
\end{align*}
\qed

\section{Higher integrability of the gradients of $\D$-minimizing functions} \label{s:higher}
In this section we prove Theorem \ref{t:higher}. As above, for the planar case we give a
simple proof which in addition provides the optimal integrability exponent.
This proof relies on the following proposition, because
by Theorem \ref{t:regularity} the singular points are isolated in dimension two.
\begin{proposition}
Let $u\in W^{1,2}(B_2,\Iq)$ be $\D$-minimizing
and assume that $\Sigma_u=\{0\}$.
Then, $|Du|\in L^p(B_1)$ for every $p< \frac{2Q}{Q-1}$.
\end{proposition}

\begin{proof}
Let $x\in B_1\setminus\{0\}$ and set $r=|x|$.
Then, by $\Sigma_u=\{0\}$, in $B_r(x)$ there exists an analytic selection
of $u$, $u\vert_{B_r(x)}=\sum_i\a{u_i}$, where $u_i:B_r(x)\to\R^{n}$
are harmonic functions.
Using the mean value inequality for $Du_i$,
one infers that
\begin{equation*}
|Du_i(x)|\leq \fint_{B_r(x)}|Du_i|\leq
\frac{1}{\sqrt{\pi}\,r}\left(\int_{B_r(x)}|Du_i|^2\right)^{\frac{1}{2}},
\end{equation*}
from which
\begin{equation}\label{e:stima grad}
|Du|(x)^2=\sum_i|Du_i(x)|^2
\leq\frac{1}{\pi\,r^2}\sum_i\int_{B_r(x)}|Du_i^2|
=\frac{\D(u,B_r(x))}{\pi\,r^2}.
\end{equation}
Using the decay estimate \eqref{e:dec} with $\rho=1$ together with \eqref{e:stima grad},
we deduce that
\begin{equation*}\label{e:est2d}
|Du|(x)\leq \frac{\D(u,B_2)}{\sqrt{\pi}\,r^{1-\frac{1}{Q}}},
\end{equation*}
which in turn implies the conclusion,
\[
\int_{B_1}|Du|^p\leq C \int_{B_1}\frac{1}{|x|^{p-\frac{p}{Q}}}
<+\infty,\quad\forall\;p<\frac{2Q}{Q-1}.
\]
\end{proof}

\begin{remark}
The range $[2,{2\,Q}{(Q-1)^{-1}})$ for the integrability
exponent is optimal. Consider, indeed, the complex variety
$\sV_Q=\{(z,w):w^Q=z\}\subseteq \C^2$.
By Theorem \ref{t:complex},
the $Q$-valued function $u(z)=\sum_{w^Q=z}\a{w}$
is $\D$-minimizing in $B_2$.
Moreover, $|Du|(z)=Q\,|z|^{\frac{1}{Q}-1}$. Hence,
$|Du|\in L^p$ for every $p<\frac{2Q}{Q-1}$ and
$|Du|\notin L^{\frac{2Q}{Q-1}}$.
\end{remark}

Now we pass to the proof of Theorem \ref{t:higher} for $m\geq3$.
The first step is a Caccioppoli's inequality for $\D$-minimizing functions.
For $P\in\R^n$, we denote by $\tau_P$ the following map:
$\tau_P:\Iqs\to\Iqs$,
\begin{equation*}\label{e:translation}
\tau_P (T) := \sum_i \a{T_i-P},\quad\text{for every}\quad
T=\sum_i \a{T_i}.
\end{equation*}

\begin{lemma}[Caccioppoli's inequality]\label{l:caccioppoli}
Let $u\in W^{1,2}(\Omega,\Iq)$ be $\D$-minimizing.
Then, for every $P\in\R^{n}$ and every $\eta\in C_c^\infty(\Omega)$,
\begin{equation}\label{e:caccioppoli}
\int_\Omega |Du|^2\,\eta^2\leq\int_{\Omega}\abs{\tau_Pu}^2\,|D\eta|^2.
\end{equation}
In particular, in the case $\Omega=B_{2r}$,
\begin{equation}\label{e:caccioppoli2}
\int_{B_{\frac{3r}{2}}} |Du|^2\leq\frac{4}{r^2}\int_{B_{2r}}\abs{\tau_Pu}^2.
\end{equation}
\end{lemma}

\begin{proof}
Recall the outer variation \cite[Proposition 3.1]{DLSp} for $\D$-minimizing functions,
\begin{equation*}
0=\int \sum_i \big\langle D f_i (x): 
D_x \psi (x, f_i (x))\big\rangle\, dx
+ \int \sum_i \big\langle Df_i (x) : 
D_y \psi (x, f_i (x))\cdot
Df_i (x)\rangle\, d x,
\end{equation*}
and  apply it to $\psi(x,y)=\eta(x)^2\,(y-P)$, where $P$ and $\eta$ are as in the statement.
Since $ D_x\psi(x,y)=2\,\eta(x)\,D\eta(x)\otimes(y-P)$ and
$D_y\psi(x,y)=\eta(x)^2\,\Id_n$,
this leads to
\begin{equation}\label{e:OV}
0=\int_\Omega \sum_i \big\langle D u_i (x) : 
2\,\eta\,D\eta \otimes (u_i-P)\big\rangle
+ \int_\Omega \sum_i \big\langle Du_i (x) :
\eta^2\, Du_i (x)\rangle.
\end{equation}
Applying H\"older's inequality in \eqref{e:OV}, we conclude \eqref{e:caccioppoli}:
\begin{align*}
\int_\Omega\eta^2\,|Du|^2&=
-\sum_i\int_\Omega\big\langle Du_i\cdot (u_i-P),\eta\,D\eta\big\rangle
\leq\int_\Omega \sum_i |Du_i|\,|u_i-P|\,|\eta|\,|D\eta|\notag\\
&\leq\int_\Omega\left(\sum_i|Du_i|^2\,|\eta|^2\right)^{\frac{1}{2}}
\left(\sum_i|u_i-P||D\eta|^2\right)^{\frac{1}{2}}\notag\\
&\leq \left(\int_\Omega\eta^2\,|Du|^2\right)^{\frac{1}{2}}
\left(\int_\Omega|\tau_P(u)|^2\,|D\eta|^2\right)^{\frac{1}{2}}.
\end{align*}
The last conclusion of the lemma follows from \eqref{e:caccioppoli}
choosing $\eta\equiv 1$ in $B_{3r/2}$ and $|D\eta|\leq \frac{2}{r}$.
\end{proof}

The following reverse H\"older inequality is the basic estimate for the higher integrability.

\begin{proposition}\label{p:est2p}
Let $\frac{2\,m}{m+2}<s<2$.
Then, there exists $C>0$ such that, for every
$u:\Omega\to\Iq$ $\D$-minimizing, $x\in\Omega$ and
$r<\min\big\{1,\dist(x,\de\Omega)/2\big\}$,
\begin{equation}\label{e:est2p}
\left(\fint_{B_r(x)}|Du|^2\right)^{\frac{1}{2}}\leq
C\left(\fint_{B_{2r}(x)}|Du|^s\right)^{\frac{1}{s}}.
\end{equation}
\end{proposition}

\begin{proof}
The proof is divided into two steps.

\textit{Step $1$: we assume that $u$ has average $0$, $\etaa\circ u=\frac{\sum_i u_i}{Q}=0$.}

The proof is by induction on the number of values $Q$.
The basic step $Q=1$ is clear: indeed, in this case $\etaa\circ u=u=0$.
Now, we assume that \eqref{e:est2p} holds for every $Q'<Q$ and, by
contradiction, it does not hold for $Q$.

Then, up to translations and dilations of the domain, there exists a sequence
$(u_l)_l\subset W^{1,2}(B_{4},\Iq)$ of $\D$-minimizing functions such that $\etaa\circ u_l=0$
and
\begin{equation}\label{e:contra}
\left(\fint_{B_{4}}|Du_l|^s\right)^{\frac{1}{s}}<
\frac{\left(\fint_{B_2}|Du_l|^2\right)^{\frac{1}{2}}}{l}.
\end{equation}
Moreover, without loss of generality, we may also assume that $\int_{B_4}|u_l|^2=1$.
Using Caccioppoli's inequality \eqref{e:caccioppoli2}, we have
that $\D(u_l,B_{3})\leq 4$,
which in turn, by \eqref{e:contra}, implies
$$
\norm{\cG(u_l,Q\a{0})}{W^{1,s}(B_4)}\leq C<+\infty.
$$
Since $s^*>2$, we can apply the compact Sobolev
embedding (see \cite[Proposition 2.11]{DLSp}) to deduce
that there exists a subsequence (not relabeled)
$u_l$ converging to some $u$ in $L^2(B_4)$.
From \eqref{e:contra} and Lemma \ref{l:sc}, we deduce that
\begin{equation}\label{e:abs1}
\int_{B_4}|u|^2=1
\qquad \text{and}\qquad \int_{B_{4}}|Du|^s=0, 
\end{equation}
which implies that $u$ is constant, $u\equiv T\in \Iq$.
Since by Theorem \ref{t:regularity} the $u_l$'s are equi-bounded
and equi-H\"older in $B_{2}$,
always up to a subsequence (again not relabeled),
the $u_l's$ converge uniformly to $T$ in $B_2$.
This implies, in particular, that
\begin{equation}\label{e:abs2}
\etaa\circ T=\lim_{l\to+\infty}\etaa\circ u_l=0.
\end{equation}
From \eqref{e:abs1} and \eqref{e:abs2}, one infers that $T$ is not a point of
multiplicity $Q$.
Therefore, since $u_l\to T$ uniformly in $B_2$, for $l$ large enough
the $u_n$'s must split in the sum of two $\D$-minimizing
functions $u_l=\a{v_l}+\a{w_l}$,
where the $v_l$'s are $Q_1$-valued functions and the $w_l$'s are
$Q_2$-valued, with $Q_1$, $Q_2$ positive and $Q_1+Q_2=Q$.
Applying now the inductive hypothesis to $v_l$ and $w_l$ we contradict
\eqref{e:contra} for $l$ large enough, 
\begin{align*}
\left(\fint_{B_1(x)}|Du_l|^2\right)^{\frac{1}{2}}
&\leq\left(\fint_{B_1(x)}|Dv_l|^2\right)^{\frac{1}{2}}
+\left(\fint_{B_1(x)}|Dw_l|^2\right)^{\frac{1}{2}}\notag\\
&\leq
C\left(\fint_{B_{2}(x)}|Dv_l|^s\right)^{\frac{1}{s}}+
C\left(\fint_{B_{2}(x)}|Dw_l|^s\right)^{\frac{1}{s}}
\notag\\
&\leq
2\,C\left(\fint_{B_{2}(x)}|Du_l|^s\right)^{\frac{1}{s}}.
\end{align*}

\textit{Step $2$: generic $\D$-minimizing function $u$.}

Let $u$ be $\D$-minimizing and $\ph=\etaa\circ u$: then,
by \cite[Lemma 3.23]{DLSp}, $\ph:\Omega\to\R^{n}$ is harmonic
and $D\ph=\sum_i Du_i$, from which
\begin{equation}\label{e:|Dph|}
|D\ph|^2\leq Q\sum_i|Du_i|^2=Q\,|Du|^2.
\end{equation}
Moreover, again by \cite[Lemma 3.23]{DLSp}, the $Q$-valued
function $v=\sum_i\a{u_i-\ph}$ is $\D$-minimizing as well.
Note that
\begin{equation}\label{e:|Dv|}
|Du|^2\leq2\,|Dv|^2+2\,Q\,|D\ph|^2
\quad\text{and}\quad
|Dv|^2\leq 2\,|Du|^2+2\,Q\,|D\ph|^2.
\end{equation}
Using the inequality $\sqrt{\sum_j a_j}\leq \sum_j\sqrt{a_j}$ for positive $a_j$,
we deduce
\begin{align}\label{e:est2p inizio}
\left(\fint_{B_r(x)}|Du|^2\right)^{\frac{1}{2}}&\leq
\left(\fint_{B_r(x)}2\,|Dv|^2+2\,Q\,|D\ph|^2\right)^{\frac{1}{2}}\notag\\
&\leq 2\left(\fint_{B_r(x)}|Dv|^2\right)^{\frac{1}{2}}+
2\,Q\left(\fint_{B_r(x)}|D\ph|^2\right)^{\frac{1}{2}}.
\end{align}
For the first term in the right hand side of \eqref{e:est2p inizio}, we use Step $1$,
since $\etaa\circ v=0$, to get
\begin{align}\label{e:term1}
\left(\fint_{B_r(x)}|Dv|^2\right)^{\frac{1}{2}} &\leq
C\left(\fint_{B_{2r}(x)}|Dv|^s\right)^{\frac{1}{s}}
\;\stackrel{\mathclap{\eqref{e:|Dv|}}}{\leq}\; C\left(\fint_{B_{2r}(x)}\left(2\,|Du|^2+2\,Q\,|D\ph|^2\right)^{\frac{s}{2}}\right)^{\frac{1}{s}}\notag\\
&\leq
C\left(\fint_{B_{2r}(x)}2\,|Du|^s+2\,Q\,|D\ph|^s\right)^{\frac{1}{s}}
\stackrel{\eqref{e:|Dph|}}{\leq}
C\left(\fint_{B_{2r}(x)}|Du|^s\right)^{\frac{1}{s}}.
\end{align}
For the remaining term in \eqref{e:est2p inizio}, we use the standard
estimate for harmonic functions,
\begin{equation}\label{e:harmonic}
|D\ph(x)|\leq\frac{C}{r^n}\,\norm{D\ph}{L^1(B_{2r})}
\qquad\forall\;x\in B_r,
\end{equation}
and infer
\begin{align}\label{e:est2p fine}
\left(\fint_{B_r(x)}|D\ph|^2\right)^{\frac{1}{2}}
&\;\;\stackrel{\mathclap{\eqref{e:harmonic}}}{\leq}\;\;
\frac{C}{r^n}\,\norm{D\ph}{L^1(B_{2r})}
\leq \frac{C}{r^n}
\left(\int_{B_{2r}(x)}|D\ph|^s\right)^{\frac{1}{s}}
\,r^{n\left(1-\frac{1}{s}\right)}\notag\\
&\;\;\leq\; C\left(\fint_{B_{2r}(x)}|D\ph|^s\right)^{\frac{1}{s}}
\;\stackrel{\mathclap{\eqref{e:|Dph|}}}{\leq}\; C\left(\fint_{B_{2r}(x)}|Du|^s\right)^{\frac{1}{s}}.
\end{align}
Clearly, \eqref{e:est2p inizio}, \eqref{e:term1} and
\eqref{e:est2p fine} finish the proof.
\end{proof}

The proof of Theorem \ref{t:higher}
is now an easy consequence of the following reverse H\"older
inequality with increasing supports proved by Giaquinta and Modica in
\cite[Proposition 5.1]{GiMo}.

\begin{theorem}[Reversed H\"older inequality]\label{t:GiMo}
Let $\Omega\subseteq\R^{m}$ be open and $g\in L_{loc}^q(\Omega)$,
with $q>1$ and $g\geq0$.
Assume that there exist positive constants $b$ and $R$ such that
\begin{equation}\label{e:hyp rh}
\left(\fint_{B_r(x)}g^q\right)^{\frac{1}{q}}\leq
b\fint_{B_{2r}(x)}g,
\quad \forall\;
x\in\Omega,\;\forall\;
r<\min\big\{R,\dist(x,\de\Omega)/2\big\}.
\end{equation}
Then, there exist $p=p(q,b)>q$ and $c=c(m,q,b)$ such that
$g\in L^p_{loc}(\Omega)$ and
\begin{equation*}\label{e;th rh}
\left(\fint_{B_r(x)}g^p\right)^{\frac{1}{p}}\leq
c\left(\fint_{B_{2r}(x)}g^q\right)^{\frac{1}{q}},
\quad\forall\;
x\in\Omega,\;\forall\;
r<\min\big\{R,\dist(x,\de\Omega)/2\big\}.
\end{equation*}
\end{theorem}

\begin{proof}[Proof of Theorem \ref{t:higher}]
Consider the function $g=|Du|^{s}$, where $s<2$ is the exponent in Proposition \ref{p:est2p}.
Estimate \eqref{e:est2p} implies that
hypothesis \eqref{e:hyp rh} of Theorem \ref{t:GiMo} is satisfied with $q=\frac{2}{s}>1$.
Hence, there exists an exponent $p'>q$,
such that $g$ belongs to $L^{p'}_{loc}(\Omega)$, i.e.~$|Du|\in L^p_{loc}(\Omega)$
for $p=p'\cdot s>2$.
\end{proof}

\bibliographystyle{abbrv}
\bibliography{reference}

\begin{thebibliography}{1}

\bibitem{Alm}
F.~J. Almgren, Jr.
\newblock {\em Almgren's big regularity paper}, volume~1 of {\em World
  Scientific Monograph Series in Mathematics}.
\newblock World Scientific Publishing Co. Inc., River Edge, NJ, 2000.
\newblock $Q$-valued functions minimizing Dirichlet's integral and the
  regularity of area-minimizing rectifiable currents up to codimension 2, With
  a preface by Jean E.\ Taylor and Vladimir Scheffer.

\bibitem{DLFS}
C.~De~Lellis, M.~Focardi, and E.~N. Spadaro.
\newblock Quasi-convexity and semicontinuity of {$Q$}-integrands.
\newblock {\em Preprint}, 2009.

\bibitem{DLSp}
C.~De~Lellis and E.~N. Spadaro.
\newblock Q-valued functions revisited.
\newblock {\em Accepted for Memoirs AMS}, 2008.

\bibitem{DLSp2}
C.~De~Lellis and E.~N. Spadaro.
\newblock Higher integrability and approximation of minimal currents.
\newblock {\em Preprint}, 2009.

\bibitem{EG}
L.~C. Evans and R.~F. Gariepy.
\newblock {\em Measure theory and fine properties of functions}.
\newblock Studies in Advanced Mathematics. CRC Press, Boca Raton, FL, 1992.

\bibitem{Fed2}
H.~Federer.
\newblock Some theorems on integral currents.
\newblock {\em Trans. Amer. Math. Soc.}, 117:43--67, 1965.

\bibitem{GiMo}
M.~Giaquinta and G.~Modica.
\newblock Regularity results for some classes of higher order nonlinear
  elliptic systems.
\newblock {\em J. Reine Angew. Math.}, 311/312:145--169, 1979.

\bibitem{Sim}
L.~Simon.
\newblock {\em Lectures on geometric measure theory}, volume~3 of {\em
  Proceedings of the Centre for Mathematical Analysis, Australian National
  University}.
\newblock Australian National University Centre for Mathematical Analysis,
  Canberra, 1983.

\end{thebibliography}
\end{document}